\documentclass[12]{amsart}
\usepackage{amsmath,amsfonts,amssymb,amsthm,enumerate,color,amscd}
\usepackage{pdfsync}
\usepackage[colorlinks=true]{hyperref}

\def\cents{\hbox{\rm\rlap/c}}

\newcommand{\thmref}[1]{Theorem~\ref{#1}}

\newcommand{\propref}[1]{Proposition~\ref{#1}}

\newcommand{\secref}[1]{\S\ref{#1}}

\newcommand{\eqnref}[1]{~(\ref{#1})}

\swapnumbers
\newtheorem{thm}{Theorem}[subsection]
\newtheorem{lem}[thm]{Lemma}
\theoremstyle{definition}
\newtheorem{cor}[thm]{Corollary}
\newtheorem{definition}[thm]{Definition}

\newtheorem{prop}[thm]{Proposition}

\theoremstyle{rem}
\numberwithin{equation}{section}

\title{}
\begin{document}

\author[]{Ben L.  Cox }\author[]{Vyacheslav Futorny }\author[]{Renato A. Martins}
\title[]{Virasoro action on Imaginary Verma modules and the operator form of the KZ-equation}

\address{Department of Mathematics \\
College of Charleston \\
66 George Street  \\
Charleston SC 29424, USA}\email{coxbl@cofc.edu}
\address{Instituto de Matem\'atica e Estat\'istica, Universidade de S\~ao Paulo, Rua do Mat\~ao,
1010, S\~ao Paulo, Brasil}\email{futorny@ime.usp.br}
\address{Instituto de Matem\'atica e Estat\'istica, Universidade de S\~ao Paulo, Rua do Mat\~ao,
1010, S\~ao Paulo, Brasil}\email{renatoam@ime.usp.br}

\begin{abstract}We define the Virasoro algebra action on imaginary Verma modules for affine $\mathfrak{sl}(2)$ and construct the analogs of Knizhnik-Zamolodchikov 
equation in the operator form. Both these results are based on a free field realization of imaginary Verma modules. 
\end{abstract}

\maketitle

\section{Introduction}\quad
The root system of an affine Kac-Moody algebra has a standard
partition
into positive and negative roots.
Corresponding to this partition is a standard Borel subalgebra,
from which one may induce the standard Verma modules.  However,
 it is possible to take other closed partitions of the root system,
and form the corresponding non-standard Borel subalgebras.  For
finite-dimensional simple Lie algebras,
we discover nothing new, but for affine Kac-Moody algebras, the
induced Verma-type modules typically contain both finite and
infinite-dimensional weight spaces.
The classification of closed subsets of the root system for affine Kac-Moody
algebras was obtained by Jakobsen and Kac \cite{JK}, and  by
Futorny \cite{F1, F2}. In this paper we are concern with 
   the affine Lie algebra of type $A_1^{(1)}$. In this case the only non-standard modules of
Verma-type
are the {\it imaginary Verma modules}. Their structure was fully understood in \cite{C1}, \cite{F1}, \cite{FS}.

In \cite{C2} a free field realization for  imaginary Verma modules was constructed for type $A$. It was then generalized to other partitions of the root system of type $A$ \cite{CF}, \cite{CF05}. These 
free field realizations (called \emph{imaginary Wakimoto modules}) generically  are isomorphic to imaginary Verma modules which resembles the connection between the standard Verma and Wakimoto modules. 

In this paper we proceed with the study of the properties of imaginary Wakimoto modules in the $\mathfrak{sl}(2)$ case. In particular, we discover a rather unexpected property: these modules 
admit a Virasoro algebra action resembling very much a vertex operator algebra structure. This is unexpected as we do not know how to obtain a Sugawara type construction 
for the action of the Virasoro algebra on an imaginary Verma module.  Since imaginary Verma modules and imaginary Wakimoto realizations are generically isomorphic 
(i.e. when the central charge is not zero), the action of the Virasoro algebra on the imaginary Wakimoto module can be transported over to an action on the imaginary Verma module.  
Furthermore, we construct the intertwining operators from an imaginary Verma module to the tensor product of an imaginary Verma module with a finite dimensional module
and obtain the analogues of Knizhnik-Zamolodchikov equation in the operator form.

\section{Preliminaries}
Let $e:=E_{12}$, $f:=E_{21}$ and $h:=E_{11}-E_{22}$
the usual basis for $\mathfrak g=\mathfrak{sl}(2,\mathbb C)$, where $E_{ij}$ is the standard basis for  $\mathfrak{gl}(2,\mathbb C)$. We have
\begin{equation}
\mathfrak g=\mathfrak n^+\oplus \mathfrak h\oplus\mathfrak n^-
\end{equation}
where $\mathfrak n^+=\mathbb C e$, $\mathfrak h=\mathbb Ch$ and $\mathfrak n^-=\mathbb C f$.

Now define
\begin{equation}
\hat{\mathfrak g}=\hat{\mathfrak{sl}}(2,\mathbb C) :=(\mathfrak g \otimes\mathbb{C} [t,t^{-1}])\oplus\mathbb{C} c=\hat{\mathfrak n}^+\oplus \hat{\mathfrak h}\oplus\tilde{\mathfrak n}^-
\end{equation}
where $\hat{\mathfrak n}^+=\mathfrak n^+\otimes \mathbb C[t, t^{-1}]\oplus \mathfrak h\otimes t\mathbb C[t],\ \hat{\mathfrak h}=\mathfrak h\oplus\mathbb Cc$ and 
$\hat{\mathfrak n}^-=\mathfrak n^-\otimes\mathbb C[t,t^{-1}]\oplus \mathfrak h\otimes t^{-1}\mathbb C[t^{-1}]$. 
The  algebra $\hat{\mathfrak g}$ has generators $e_m, f_m, h_m$,  $m\in\mathbb Z$, and central element $c$ with the product
$$
[X_m, Y_n]=[X,Y]_{m+n}+ \delta_{m+n,0}m(X|Y)c,
$$
where $X_m:=X\otimes t^m$ for $X,Y\in \mathfrak g$ and $m\in\mathbb Z$ and $$
(X|Y)=\text{tr}\,(XY).
$$ is the Killing form. Now define
\begin{equation}
\tilde{\mathfrak g}:=\hat{\mathfrak g}\oplus\mathbb C d=\hat{\mathfrak n}^+\oplus\tilde{\mathfrak h}\oplus\hat{\mathfrak n}^-,\end{equation}
where $\tilde{\mathfrak h}=\hat{\mathfrak h}\oplus\mathbb Cd$, $[d,X_m]=mX_m$ and $[d,c]=0$. 

A subalgebra $\mathfrak b_{nat}=\tilde{\mathfrak h}\oplus \hat{\mathfrak n}^+$ is the \emph{natural Borel subalgebra} introduced by Jakobsen and Kac \cite{JK}.


In $\tilde{\mathfrak h}$, we have the dual space $\tilde{\mathfrak h}^*$ such that
\begin{equation}\label{dual}
\tilde{\mathfrak h}^*=\mathfrak h^*\oplus\mathbb C\Lambda_0\oplus\mathbb C\delta
\end{equation}
where for all $h\in\mathfrak h$ we have
$$
\Lambda_0(c)=1,\ \Lambda_0(d)=\Lambda_0(h)=0,
$$
$$
\delta(d)=1,\ \delta(c)=\delta(h)=0.
$$

Let $\Delta$ denote the root system of $\tilde{\mathfrak g}$, and let
$\{ \alpha_0, \alpha_1\}$ be a basis for $\Delta$.  Then $\delta = \alpha_0 + \alpha_1$ is 
the indivisable positive imaginary root  and
$$
\Delta = \{ \pm \alpha_1 + n\delta\ |\ n \in \mathbb Z\} \cup \{ k\delta\ |\ k \in \mathbb Z
\setminus \{ 0 \} \}.
$$

\subsection{Imaginary Verma Modules}

Let $\Lambda\in\tilde{\mathfrak h}^*$ be a weight and let $\mathbb C_\Lambda$ be a one-dimensional representation 
of $\tilde{\mathfrak h}\oplus\hat{\mathfrak n}^+$ such that $\mathbb C_\Lambda$ is generated by a vector $v_\Lambda$, 
$\hat{\mathfrak n}^+$ acts by zero and $\tilde{h}v_\Lambda=\Lambda(\tilde{h})v_\Lambda$, for all $\tilde{h}\in\tilde{\mathfrak h}$. 

The \emph{imaginary Verma module} is defined as follows (\cite{F1}, \cite{C1}):
\begin{equation}
V_\Lambda=\text{Ind}_{\tilde{\mathfrak h}\oplus\hat{\mathfrak n}^+}^{\tilde{\mathfrak g}}\mathbb C_\Lambda,
\end{equation}

When a representation $V$ of $\tilde{\mathfrak g}$ is generated by a vector $v_\Lambda$ of weight $\Lambda$ that is annihilated by 
$\hat{\mathfrak n}^+$, we say that $V$ has the (imaginary) highest weight $\Lambda$. For any weight $\Lambda$ we can write
\begin{equation}\label{Lambda}
\Lambda=\lambda+\kappa\Lambda_0-\Delta\delta\text{, }\lambda\in\mathfrak h^*\text{ and }\kappa,\Delta\in\mathbb C.\end{equation}
where by \eqnref{dual}, we have $\lambda(h)=\Lambda(h)$, $\kappa=\Lambda(c)$ and $\Delta=-\Lambda(d)$. In the following we will only consider weights $\Lambda$ such that
\begin{equation}
\Delta=\Delta(\lambda)=\dfrac{\langle\lambda,\lambda+2\rho\rangle}{2(k+h^v)}
\end{equation}
and for this $\Lambda$ we will denote $V_{\lambda,\kappa}$ instead of $V_\Lambda$. When $\kappa$ and $\Delta$ are fixed we just call $V_{\lambda,\kappa}$ the imaginary 
highest weight module with imaginary highest weight $\lambda$.

Let $H=(\mathfrak h \otimes \mathbb C[t]t)\oplus (\mathfrak h \otimes \mathbb C[t^{-1}]t^{-1}) \oplus \mathbb C c$ be a Heisenberg subalgebra of $\tilde{\mathfrak g}$. 
Let $V_{\lambda}$ be the Verma $H$-module generated by $v_\Lambda$. Then
\begin{equation}\label{vlk}
V_{\lambda,\kappa}=\text{Ind}_{H+\mathfrak b_{nat}}^{\tilde{\mathfrak g}}V_\lambda,
\end{equation}
where $d$ acts on $V_\lambda$ by multiplication by $-\Delta(\lambda)$. In $V_{\lambda,\kappa}$ we have a $\mathbb Z$-grading:
$$
V_{\lambda,\kappa}=\bigoplus_{n\geq 0}V_{\lambda,\kappa}[-n].
$$

Observe that $V_{\lambda,\kappa}[-n]$ is the eigenspace of $d$ with the eigenvalue $-n-\Delta(\lambda)$, each $V_{\lambda,\kappa}[-n]$ is an $H$-module  and $V_{\lambda,\kappa}[0]=V_\lambda$.

Let $V$ be a $\mathfrak g$-module and fix $z\in\mathbb C^*$. For any polynomial $P(t)\in\mathbb C[t]$ and $x\in\hat{\mathfrak g}$, $u\in V$ set
\begin{equation}\label{evaluationrepresentation}
x\otimes P(t)\cdot u=P(z)xu,\ cu=0
\end{equation}

This is called the evaluation representation of $\hat{\mathfrak g}$ and we denote it by $V(z)$. Unfortunately we can't extend it to an action 
of $\tilde{\mathfrak g}$. Let $\Delta$ be a convenient complex number (that we will define in \propref{iff}) and $z$ a formal variable. Consider the space
$$
z^{-\Delta}V[z,z^{-1}]=V\otimes z^{-\Delta}\mathbb C[z,z^{-1}].
$$
 This space has a module structure for  $\tilde{\mathfrak g}$, where  $d$ acts by $z\frac{\partial}{\partial_z}$.  For any $z_0\neq 0$, we have the evaluation map
\begin{equation}
\epsilon_{z_0}:z^{-\Delta}V[z,z^{-1}]\longrightarrow V(z_0)
\end{equation}
that is a $\hat{\mathfrak g}$-epimorphism.
 Denote by
\begin{equation}\label{completedtensorproduct}
V_{\lambda,\kappa}\hat{\otimes}V(z_0)
\end{equation}
the completed tensor product generated by all infinite expressions of the form $\displaystyle\sum_{i=1}^\infty w_i\otimes v_i$, where $w_i\in V_{\lambda,\kappa}$, is a 
homogeneous vector, $\{degree(w_i)\}\longrightarrow-\infty$ and $v_i\in V(z_0),\ \forall i\in\mathbb N^*$.


\subsection{Formal distributions}

\begin{definition}
A formal distribution is an
expression of the form
$$
a(z,w,\dots)=\sum_{m_1,m_2,\dots\in\mathbb Z}a_{m_1,m_2,\dots}z^{m_1}w^{m_2}\dots
$$
where the $a_{m_1,m_2,\dots}$ lie in some fixed vector space $V$.

 If $A(z)$ is a field, then we set
\begin{equation}\label{minusplus}
    A(z)_-:=\sum_{m\geq 0}A_{m}z^{-m-1},\quad\text{and}\quad
   A(z)_+:=\sum_{m<0}A_{m}z^{-m-1}.
\end{equation}
The normal ordered product of two formal distributions $A(z)$ and $B(w)$ is defined by
$$
:A(z)B(w):=\displaystyle\sum_{m\in\mathbb Z}\displaystyle\sum_{n\in\mathbb Z}:A_mB_n:z^{-m-1}w^{-n-1}=A(z)_+B(w)+B(w)A(z)_-.
$$
\end{definition}

\section{Free field realizations}

\subsection{First free field realization}

To get a free field realization of the imaginary Verma module for $\hat{sl}(2)$  Bernard and Felder used the Borel-Weyl construction \cite{BF}. Let $\hat{B}_-$ be the Borel
subgroup of the loop group $\hat{SL}(2)$ corresponding to a Borel
subalgebra $\mathfrak b_{nat}$. Then $\hat{B}_-$ consists of the
elements

\begin{align*}
\exp(\sum_{n\in\mathbb Z}x_ne_n)\exp(\sum_{m>0}y_mh_m),
\end{align*}
where $x_n, y_m$ are coordinate functions. Consider a one
dimensional representation  $\chi: \hat{B}_-\rightarrow \mathbb C$, where
$c$ acts by scalar $K$,  $h$ acts by scalar $J$ ($J/2$ is called
the {\em spin}) and all other elements act trivially. Then one can
construct a line bundle over $\hat{SL}(2)/\hat{B}_-$ by taking a
fiber product
$$\mathcal L_{\chi}=\hat{SL}(2) \times_{\hat{B}_-} \mathbb C$$ and a map
$g:\hat{SL}(2) \times_{\hat{B}_-} \mathbb C \rightarrow
\hat{SL}(2)/\hat{B}_-$ such that $(x,z)\mapsto x \hat{B}_-$. 
 Differentiating the action of the group $\hat{SL}(2)$ acts on the sections of the line bundle  to an action of
the Lie algebra $\mathfrak g$ and applying two anti-involutions

$$e_n\leftrightarrow -f_{-n}, \,\, h_n\leftrightarrow h_{-n}, \,\,\, c\leftrightarrow c$$

and

$$x_{-n}\leftrightarrow \partial x_n, \, \, \, \, \, y_k\leftrightarrow -\partial y_k.$$

we obtain the following free field realization of $\mathfrak g$ in the
Fock space $\mathbb C[x_m, m\in \mathbb Z]\otimes \mathbb C[y_n, n>0]$:

$$f_n\mapsto  x_n, \,\,\, h_n\mapsto -2\sum_{m\in \mathbb Z}x_{m+n}\partial x_{m}
+\delta_{n<0}y_{-n} + \delta_{n>0}2nK\partial y_{n}
+\delta_{n,0}J,$$

$$e_n\mapsto -\sum_{m,k\in \mathbb Z}x_{k+m+n}\partial x_k\partial x_{m}
+\sum_{k>0}y_{k}\partial x_{-k-n}+ 2K\sum_{m>0}m\partial
y_m\partial x_{m-n} +(Kn+J)\partial x_{-n}.$$

 This module is irreducible if and only if $K\neq 0$. If we
 let $K=0$ and quotient out the submodule generated by $y_m, m>0$ 
then we obtain what is usually called \emph{the first free field realization} of $\hat{sl}(2)$. 

This quotient is irreducible if and only if $J\neq 0$ (cf.
\cite{F1}). This  has been generalized for all affine
Lie algebras in \cite{C2}.

\subsection{Wakimoto modules}

We recall the \emph{second free field realization} of $\hat{sl}(2)$, that is Wakimoto module construction \cite{Wak86}. 

Set
$$
a^*(z):=\sum_{m\in\mathbb Z}a^*_{m}z^{-m},\quad
a(z):=\sum_{m\in\mathbb Z}a_{m}z^{-m-1},
\quad b(z):=\sum_{m\in\mathbb Z}b_{m}z^{-m-1}.
$$

Let now $$a_n=\{
\begin{array}{cc}
x_n,  & n<0\\
\partial x_n,  & n\geq 0,
\end{array}
\,\,\, a^*_n=\{
\begin{array}{cc}
x_{-n},  & n\leq0\\
-\partial x_{-n},  & n>0,
\end{array}
\,\,\, b_m=\{
\begin{array}{cc}
m\partial y_m,  & m\geq 0\\
y_{-m},  & m<0.
\end{array}
$$
Here  $[a_n, a^*_m]=\delta_{n+m,0}$ and $[b_n, b_m]=n\delta_{n+m,0}$.

\begin{thm}(\cite{Wak86}).
The formulas
$$c\mapsto K, \,\,\, e(z)\mapsto a(z), \,\,\, h(z)\mapsto -2:a^*(z)a(z): +b(z),$$
$$f(z)\mapsto -:a^*(z)^2 a(z): +K\partial_za^*(z) +a^*(z)b(z)$$
define the {\em second free field realization} of the affine
$sl(2)$ acting on the space\\ $\mathbb C[x_n, n\in \mathbb Z]\otimes \mathbb C[y_m,
m>0]$.
\end{thm}

These modules are celebrated {\em Wakimoto modules}. They
were defined for an arbitrary affine Lie algebra by Feigin and
Frenkel \cite{FF1}, \cite{FF2}. 
Generically Wakimoto modules   are isomorphic to Verma modules. 

\subsection{Imaginary Verma modules revisited}

We will re-write the first free field realization in terms of the ``fields". 

Let $\hat{\mathfrak a}$ be the infinite dimensional Heisenberg
algebra with generators $a_{m}$, $a_{m}^*$, and $\mathbf 1$,
 $m\in \mathbb Z$, subject to the
relations
\begin{align*}
[a_{m},a_{n}]&=[a_{m}^*,a^*_{n}]=0, \\
[a_{m},a^*_{n}]&=
\delta_{m+n,0}\mathbf 1, \\
[a_{m},\mathbf 1]&=[a^*_{m},\mathbf 1]=0.
\end{align*}
This algebra acts on $\mathbb C[x_{m}|m\in \mathbb Z
]$
   by
$$
 a_{m}\mapsto x_{m},
 \quad a_{m}^* \mapsto
 -\partial/\partial
x_{-m}$$
and $\mathbf 1$ acts as an identity. Hence we have an
  $\hat{\mathfrak a}$-module generated by $v$ such that
$$
a_{m}^*v=0, m\in \mathbb Z.
$$

Observe that $a(z)$  is not a field whereas $a^*(z)$
is  a field.   We will call $a(z)$ (resp. $a^*(z)$)
a {\it pure creation} (resp. {\it annihilation}) {\it operator}. 
Set 
\begin{align*}
a(z)_+&=a(z),\quad a(z)_-=0,\\
a^*(z)_+&=0,\quad a^*(z)_-=a^*(z).
\end{align*}

Define
$$
e(z)=\sum_{n\in\mathbb Z}e_{n}z^{-n-1},\quad
f(z)=\sum_{n\in\mathbb Z}f_{n}z^{-n-1},\quad
h(z)=\sum_{n\in\mathbb Z} h_{n}z^{-n-1}.$$

\begin{thm}(\cite{CF05}).
Let $\lambda\in \mathfrak H^*$, $\gamma\in \mathbb C$.
The generating functions

\begin{equation}
\begin{array}{l}
f(z) \mapsto a(z),\\
h(z) \mapsto 2:a(z)a^*(z):+ b(z), \\
e(z) \mapsto :a^*(z)^2a(z):
-a^*(z)b(z)-
     \left(1-\gamma^2\right)\partial_z a^*(z), \\
c\mapsto \gamma^2-1,
\end{array}
\end{equation}
define a representation  $\rho: \hat{\mathfrak{sl}}(2)\rightarrow \mathfrak{g}\mathfrak{l}(\mathbb C[\mathbf
x]\otimes \mathbb C[\mathbf y])$, where $\mathbb C[\mathbf
x]=\mathbb C[x_m, m\in \mathbb Z]$, $\mathbb C[\mathbf y]=\mathbb C[y_n, n>0]$. 
\end{thm}

In other words, the first and the second free field realizations can be obtained from the same (Wakimoto) formulas by taking different representations of the Heisenberg algebra 
 $\hat{\mathfrak a}$.

We will denote this $\hat{\mathfrak{sl}}(2)$-module by $W_{\lambda,\kappa}$, where $\kappa=\gamma^2-1$ and call it \emph{imaginary Wakimoto module}. 
We can make $W_{\lambda,\kappa}$ into a $\tilde{\mathfrak{sl}}(2,\mathbb C)$-module by defining 
$$
d\cdot a_{n_1}\cdots a_{n_k}b_{-m_1}\cdots b_{-m_l}\cdot 1=\left(\sum_{j=1}^kn_k-\sum_{i=1}^lm_i-\Delta(\lambda)\right)a_{n_1}\cdots a_{n_k}b_{-m_1}\cdots b_{-m_l}\cdot 1
$$
so that in particular $d\cdot 1=-\Delta(\lambda)\cdot1$. From now one we will let $w_{\lambda,\kappa}$ denote the generator $1$ in $W_{\lambda,\kappa}$.

\begin{thm}[\cite{CF05}] \label{cf05}
Fix $\Lambda$ as in \eqnref{Lambda}.  For $\kappa=\gamma^2-1\neq 0$,  the map $\Psi_{\lambda,\kappa}:V_{\lambda,\kappa}\longrightarrow W_{\lambda,\kappa}$ given by
$$\Psi_{\lambda,\kappa}(f_{n_1}\cdots f_{n_k}h_{-m_1}\cdots h_{-m_l}\cdot v_{\lambda,\kappa})=a_{n_1}\cdots a_{n_k}b_{-m_1}\cdots b_{-m_l}\cdot w_{\lambda ,\kappa}$$
is an isomorphism of $\tilde{\mathfrak{sl}}(2,\mathbb C)$-modules.
\end{thm}

\section{The Virasoro action}

Our first goal is to define an action of the Virasoro algebra on imaginary Verma module.  
We follow the construction done in \cite{KR87}. The Virasoro algebra is defined to have basis $L_n$, $n\in\mathbb Z$ and center spanned by $c$ with relations
$$
[L_m,L_n]:=(m-n)L_{m+n}+\frac{1}{12}(m^3-m)\delta_{m+n,0}c.
$$
Define
$$
\bar L_k:= \sum_{j\in\mathbb Z}(j-k)a_{j}a_{k-j}^*
$$
which is to be viewed as acting on the imaginary Wakimoto module. Now we have

\begin{lem}
\begin{equation}
[a_k,\bar L_n]=ka_{k+n},\quad [a_k^*,\bar L_n]=(k+n)a_{k+n}^*
\end{equation}
\end{lem}
Define $\psi:\mathbb R\to \{0,1\}$ by
$$
\psi(x):=\begin{cases}
1 &\quad \text{ if } |x|\leq 1 ; \\ 0 & \quad \text{ if } |x|> 1.
\end{cases}
$$

Then set
$$
\bar L_k(\epsilon):= \sum_{j\in\mathbb Z}(j-k)a_{j}a_{k-j}^*\psi
(\epsilon j)
$$

Now $\bar{L}_n(\epsilon)$ has only a finite number of summands and as $\epsilon\to 0$, $\bar{L}_n(\epsilon)\to \bar{L}_n$.
\begin{proof}
\begin{align*}
[a_k,\bar L_n(\epsilon)]&= \sum_{j\in\mathbb Z}(j-n)[a_k,a_{j}a_{n-j}^*]\psi(\epsilon j) = ka_{k+n}\psi(\epsilon (n+k))   \\
[a_k^*,\bar L_n(\epsilon)]&= \sum_{j\in\mathbb Z}(j-n)[a_k^*,a_{j}a_{n-j}^*] \psi(\epsilon j)=( k+n)a_{k+n}^* \psi(\epsilon (-k))
\end{align*}
\end{proof}

\begin{lem}
$[\bar L_m,\bar L_n]=(m-n)\bar L_{m+n}.$
\end{lem}
\begin{proof}
\begin{align*}
[\bar L_m(\epsilon),\bar L_n]&= \sum_{j\in\mathbb Z}(j-m)[a_{j}a_{m-j}^*,\bar L_n]\psi(\epsilon j)    \\
 &=\sum_{j\in\mathbb Z}(j-m)[a_{j},\bar L_n]a_{m-j}^*\psi(\epsilon j) +\sum_{j\in\mathbb Z}(j-m)a_{j}[a_{m-j}^*,\bar L_n]\psi(\epsilon j)  \\
 &=\sum_{j\in\mathbb Z}(j-m)ja_{j+n} a_{m-j}^*\psi(\epsilon j) +\sum_{j\in\mathbb Z}(j-m)(m+n-j)a_{j}a_{m+n-j}^* \psi(\epsilon j)  \\
 &=\sum_{j\in\mathbb Z}(j-n-m)(j-n)a_{j} a_{m+n-j}^*\psi(\epsilon (j-n)) +\sum_{j\in\mathbb Z}(j-m)(m+n-j)a_{j}a_{m+n-j}^* \psi(\epsilon j)  \\
  &=\sum_{j\in\mathbb Z}\left((j-n-m)(j-n) \psi(\epsilon (j-n)) + (j-m)(m+n-j)\psi(\epsilon j) \right)a_{j}a_{m+n-j}^*
\end{align*}

Now for $\epsilon$ small enough and for $j$ fixed we have
\begin{align*}
(j-n-m)&(j-n) \psi(\epsilon (j-n)) + (j-m)(m+n-j)\psi(\epsilon j) \\
    &=(j-n-m)(j-n)  + (j-m)(m+n-j)=(m-n)(j-m-n)
\end{align*}

Hence
$$
[\bar L_m,\bar L_n]=(m-n) \bar L_{m+n}
$$
\end{proof}

Let $\bar{L}(z)=\displaystyle\sum_{n\in\mathbb Z}\bar{L}_nz^{-n-2}$. We have
\begin{align*}
\ [\bar{L}(z),\bar{L}(w)]&=\displaystyle\sum_{k\in\mathbb Z}\displaystyle\sum_{q\in\mathbb Z}[\bar{L}_k,\bar{L}_q]z^{-k-2}w^{-q-2}\\
&=\displaystyle\sum_{k\in\mathbb Z}\displaystyle\sum_{q\in\mathbb Z}(k-q)\bar{L}_{k+q}z^{-k-2}w^{-q-2}\\
&=\displaystyle\sum_{l\in\mathbb Z}\displaystyle\sum_{q\in\mathbb Z}(l-2q)\bar{L}_lz^{-l+q-2}w^{-q-2}\\
&=\Big(\displaystyle\sum_{l\in\mathbb Z}(l+2)\bar{L}_lz^{-l-3}\Big)\Big(\displaystyle\sum_{q\in\mathbb Z}z^{q+1}w^{-q-2}\Big)\\
&\quad+\Big(\displaystyle\sum_{l\in\mathbb Z}\bar{L}_lz^{-l-2}\Big)\Big(\displaystyle\sum_{q\in\mathbb Z}(-2q-2)z^qw^{-q-2}\Big)\\
&=-\partial_z\bar{L}(z)\delta(z/w)+2\bar{L}(z)\partial_w\delta(z/w)\\
&=\partial_w\bar{L}(w)\delta(z/w)+2\bar{L}(w)\partial_w\delta(z/w).
\end{align*}
\label{mu}
Let $\mu\in \mathbb C$ be fixed. Now define $L_k$ by
\begin{align}
L(z):&=\sum_{k\in\mathbb Z}L_kz^{-k-2}=\bar L(z)+\frac{1}{4}:b(z)^2:+\frac{\mu}{2} \partial_zb(z)\\
&=\sum_{k\in\mathbb Z}\left(\sum_{j\in\mathbb Z}(j-k)a_ja_{k-j}^*+ \frac{1}{4} \sum_{j\in\mathbb Z} :b_jb_{k-j}:-\frac{\mu}{2} (k+1)b_k\right)z^{-k-2} \label{L}
\end{align}
where $\displaystyle\sum_{n\in\mathbb Z}:b_nb_{m-n}: =\displaystyle\sum_{n>m}b_{m-n}b_n+\displaystyle\sum_{n\leq m}b_nb_{m-n}$ because

\begin{align*}
:b(z)b(z):=b(z)_+b(z)+b(z)b(z)_-&=\sum_{m<0,n\in\mathbb Z}b_mb_nz^{-m-n-2}+\sum_{m\geq 0,n\in\mathbb Z}b_nb_mz^{-m-n-2} \\
&=\sum_{m'<n,n\in\mathbb Z}b_{m'-n}b_nz^{-m'-2}+\sum_{m'\geq n, n\in\mathbb Z}b_nb_{m'-n}z^{-m'-2} \\
\end{align*}
and then as we shall see below, the center of the Virasoro acts by the scalar $\cents=6-6\mu^2\in\mathbb C$.  We have $b_i$ satisfying the relation
$$
[b_m,b_p]=2m\delta_{m+p,0},
$$
so that 
$$
[b(z),b(w)]=2\partial_w\delta(z/w),
$$
$$
b(z)b(w)=:b(z)b(w):+\frac{2}{(z-w)^2}.
$$
\begin{prop}\label{bracket}
Take $b_0v=\lambda v$, $cv=\cents v$ where $v=w_{\lambda,\kappa}$ is the vacuum vector. Then
\begin{equation}
[L(z),L(w)]=\frac{c}{12}\partial_w^3\delta(z/w)+2L(w)\partial_w\delta(z/w)+\partial_wL(w)\delta(z/w).
\end{equation}
\end{prop}

\begin{proof}
We first calculate using Kac's Theorem, Wick's Theorem and Taylor's Theorem\\ \\
$\begin{aligned}
:b(z)^2::b(w)^2&=:b(z)^2b(w)^2:+\frac{8}{(z-w)^2}:b(z)b(w):+\frac{8}{(z-w)^4} \\
&=:b(z)^2b(w)^2:+\frac{8}{(z-w)^2}:b(w)b(w):+\frac{8}{(z-w)}:b(w)\partial_wb(w):+\frac{8}{(z-w)^4}
\end{aligned}$\\

Hence\\ \\
$[\frac{1}{4}:b(z)^2:,\frac{1}{4}:b(w)^2:]=2\left(\frac{1}{4}\right):b(w)b(w):\partial_w\delta(z/w)+\frac{1}{2}:b(w)\partial_wb(w):+\frac{1}{2}\partial_w^3\delta(z/w).$\\

Next we have\\ \\
$\begin{aligned}
:b(z)^2:b(w)&=:b(z)^2b(w):+\frac{4}{(z-w)^2}  b(w)+\frac{4}{z-w}  \partial_wb(w)\ 
\end{aligned}$\\ \\
so that\\ \\
$\begin{aligned}
\Big[&\frac{1}{4}:b(z)^2:,\frac{\mu}{2}\partial_wb(w)\Big]+\Big[\frac{\mu}{2}\partial_zb(z),\frac{1}{4}:b(w)^2:\Big]\\
&=\partial_w\Big(\frac{\mu}{2}b(w)\partial_w\delta(z/w)+\frac{\mu}{2}\partial_wb(w)\delta(z/w)\Big)+\partial_z\Big(\frac{\mu}{2}b(z)\partial_z\delta(z/w)+\frac{\mu}{2}\partial_zb(z)\delta(z/w)\Big)\\
&=\frac{2\mu}{2}\partial_wb(w)\partial_w\delta(z/w)+\frac{\mu}{2}\partial^2_wb(w)\delta(z/w) +\frac{\mu}{2}b(w)\partial^2_w\delta(z/w)-\frac{\mu}{2}\partial_z^2(b(z)\delta(z/w))\\
&=\mu\partial_wb(w)\partial_w\delta(z/w)+\frac{\mu}{2}\partial^2_wb(w)\delta(z/w) +\frac{\mu}{2}b(w)\partial^2_w\delta(z/w)-\frac{\mu}{2}\partial_z^2(b(w)\delta(z/w))\\
&=\mu\partial_wb(w)\partial_w\delta(z/w)+\frac{\mu}{2}\partial^2_wb(w)\delta(z/w) +\frac{\mu}{2}b(w)\partial^2_w\delta(z/w)-\frac{\mu}{2}b(w)\partial_z^2\delta(z/w)\\
&=\mu\partial_wb(w)\partial_w\delta(z/w)+\frac{\mu}{2}\partial^2_wb(w)\delta(z/w) +\frac{\mu}{2}b(w)\partial^2_w\delta(z/w)-\frac{\mu}{2}b(w)\partial_w^2\delta(z/w)
\end{aligned}$\\

Lastly we compute
$[\frac{\mu}{2}\partial_zb(z),\frac{\mu}{2}\partial_wb(w)]=-\frac{\mu^2}{2}\partial_w^3\delta(z/w).$\\

Putting all these calculations together we get\\ \\
$\begin{aligned}
\ [L(z)&,L(w)]=\partial_w\bar{L}(w)\delta(z/w)+2\bar{L}(w)\partial_w\delta(z/w)+\frac{1}{2}:b(w)b(w):\partial_w\delta(z/w)\\
&\quad+\frac{1}{2}:b(w)\partial_wb(w):+\frac{1}{2}\partial_w^3\delta(z/w)+\mu\partial_wb(w)\partial_w\delta(z/w)+\frac{\mu}{2}\partial^2_wb(w)\delta(z/w)\\
&\quad +\frac{\mu}{2}b(w)\partial^2_w\delta(z/w)-\frac{\mu}{2}b(w)\partial_w^2\delta(z/w)-\frac{\mu^2}{2}\partial_w^3\delta(z/w)\\
&=\Big(\partial_w\bar{L}(w)+\frac{1}{2}:b(w)\partial_wb(w):+\frac{\mu}{2}\partial_w^2b(w)\Big)\delta(z/w)\\
&\quad+\Big(2\bar{L}(w)+\frac{1}{2}:b(w)^2:+\frac{\mu}{2}\partial_wb(w)\Big)\partial_w\delta(z/w)+\Big(\frac{1}{2}-\frac{\mu^2}{2}\Big)\partial_w^3\delta(z/w)\\
&=\partial_wL(w)\delta(z/w)+2L(w)\partial_w\delta(z/w)+\frac{c}{12}\partial_w^3\delta(z/w)
\end{aligned}$\\
\end{proof}

Next we have\\
\begin{prop}
$$
[a(z),L(w)]= a(w)\partial_w\delta(z/w),
$$
$$
[a^*(z),L(w)]=-\partial_wa^*(w)\delta(z/w),
$$
$$
[b(z),L(w)]= b(w)\partial_w\delta(z/w)+\mu\partial_w^2\delta(z/w).
$$
\end{prop}

\begin{proof}
We have\\ \\
$\begin{aligned}
\ [a(z),L(w)]&=[a(z),\bar{L}(w)]=\sum_{k,n\in\mathbb Z}[a_k,\bar{L}_n]z^{-k-1}w^{-n-2}=\sum_{k,n\in\mathbb Z}ka_{k+n}z^{-k-1}w^{-n-2}\\
&=\sum_{k,n\in\mathbb Z}ka_{k+n}z^{-k-1}w^{k-1}w^{-n-k-1}=\sum_{k,n\in\mathbb Z}ka_{n-k}z^{k-1}w^{-k-1}w^{-n+k-1}\\
&= a(w)\partial_w\delta(z/w),
\end{aligned}$\\ \\
$\begin{aligned}
\ [a^*(z),L(w)]&= [a^*(z),\bar{L}(w)]=\sum_{k,n\in\mathbb Z}[a_k^*,\bar{L}_n]z^{-k}w^{-n-2}=\sum_{k,n\in\mathbb Z}(k+n)a_{k+n}^*z^{-k}w^{-n-2}\\
&=\sum_{k,n\in\mathbb Z}(k+n)a_{k+n}^*z^{-k}w^{k}w^{-1}w^{-n-k-1}  \\
        &=- \partial_wa^*(w)\delta(z/w)
\end{aligned}$\\ \\
and using calculations from \propref{bracket} we have\\ \\
$\begin{aligned}
\ [b(z),L(w)]&=[b(z),\frac{1}{4}:b(w)^2:+\frac{\mu}{2}\partial_wb(w) ] \\
        &=\frac{1}{4}\left(-4\partial_zb(z)\delta(z/w)-4b(z)\partial_z\delta(z/w)\right)+\mu\partial_w^2\delta(z/w)  \\
        &= -\partial_wb(w)\delta(z/w)+b(z)\partial_w\delta(z/w)+\mu\partial_w^2\delta(z/w)\\
        &= b(w)\partial_w\delta(z/w)+\mu\partial_w^2\delta(z/w).
\end{aligned}$

\end{proof}

Note that the above implies 
\begin{align}
[L_{-1},a(z)]&=\partial_za(z),\ [L_{-1},a^*(z)]=\partial_za^*(z),\ [L_{-1},b(z)]=\partial_zb(z),\\
[L_{0},a(z)]&=a(z)+z\partial_za(z),\ [L_{0},a^*(z)]=z\partial_za^*(z),\ [L_{0},b(z)]=b(z)+z\partial_zb(z).\label{L0}
\end{align}

Indeed for example from the proof we have
$$
[L_{-1},a^*(z)]=-\sum_{k\in\mathbb Z}(k-1)a_{k-1}^*z^{-k}=-\sum_{s\in\mathbb Z}sa_{s}^*z^{-s-1}=\partial_za^*(z).
$$

Note also that 
\begin{equation}
L(z)=a(z)\partial_za^*(z)+\frac{1}{2}:b(z)^2:+\mu \partial_zb(z).
\end{equation}

\begin{prop}
\begin{align*}
[\rho(f(z)),L(w)]&= \rho(f(w))\partial_w\delta(z/w) \\
[\rho(h(z)),L(w)]&=\rho(h(w))\partial_w\delta(z/w)+\mu\partial_w^2\delta(z/w)\\ 
[\rho(e(z)),L(w)]&=\rho(e(w))\partial_w\delta(z/w)  -\mu\partial_w^2\left(a^*(w)\delta(z/w)  \right)
\end{align*}
\end{prop}

\begin{proof}
Here are the calculations:\\ \\
$\begin{aligned}
\ [\rho(f(z)),L(w)]&=[a(z), L(w)]=a(w)\partial_w\delta(z/w)= \rho(f(w))\partial_w\delta(z/w)
\end{aligned}$\\ \\
$\begin{aligned}
\ [\rho(h(z)),L(w)]&=[2a(z)a^*(z)+b(z),L(w)] \\
&=-2a(z)(\partial_wa^*(w))\delta(z/w)+2a(w)a^*(w)\partial_z\delta(z/w)+[b(z), L(w)]\\
&= 2a(w)(-\partial_wa^*(w)\delta(z/w)+\partial_w(a^*(w)\delta(z/w))+[b(z), L(w)] \\ &=2a(w)a^*(w)\partial_w\delta(z/w) +b(w)\partial_w\delta(z/w)+\mu\partial_w^2\delta(z/w)\\
&=\rho(h(w))\partial_w\delta(z/w)+\mu\partial_w^2\delta(z/w)
\end{aligned}$\\ \\
$\begin{aligned}
\ [\rho(e(z)),L(w)]&= -[a(z)a^*(z)^2 + a^*(z)b(z)+(1-\gamma^2)\partial_za^*(z), L(w)]\\
&=-a(w)a^*(z)^2\partial_w\delta(z/w)+2a(z)a^*(z)\partial_wa^*(w)\delta(z/w)\\
&\quad-(b(w)\partial_w\delta(z/w)+\mu\partial_w^2\delta(z/w))a^*(z)+b(z)\partial_wa^*(w)\delta(z/w) \\
&\quad  +(1-\gamma^2)\partial_z(\partial_wa^*(w)\delta(z/w))\\
&=-a(w)a^*(w)^2\partial_w\delta(z/w)-b(w)a^*(w) \partial_w\delta(z/w)\\
&\quad-\mu\partial_w^2\delta(z/w)a^*(z)+ (1-\gamma^2)\partial_wa^*(w)\partial_z\delta(z/w)\\
&=-a(w)a^*(w)^2\partial_w\delta(z/w)-b(w)a^*(z) \partial_w\delta(z/w)\\
&\quad-\mu\Big(a^*(w)\partial_w^2\delta(z/w)+2\partial_wa^*(w)\partial_w\delta(z/w)+\partial_w^2a^*(z)\delta(z/w)\Big)\\
&\quad-(1-\gamma^2)\partial_wa^*(w)\partial_w\delta(z/w)\\
&=\rho(e(w))\partial_w\delta(z/w)  -\mu\left(a^*(w)\partial_w^2\delta(z/w) +2\partial_wa^*(w)\partial_w\delta(z/w)  +\partial_w^2a^*(z)\delta(z/w)  \right)\\
&=\rho(e(w))\partial_w\delta(z/w)  -\mu\partial_w^2\left(a^*(w)\delta(z/w)  \right)
\end{aligned}$

\end{proof}
 
\begin{cor}
$$
[L_{-1},\rho(f(z))]=\partial_z\rho(f(z)),\quad  [L_{-1},\rho(e(z))]=\partial_z\rho(e(z)),
$$
$$
[L_{-1},\rho(h(z))]=\partial_z\rho(h(z)),\quad  [L_{-1},b(z)]=\partial_zb(z),
$$
$$
[L_{0},\rho(f(z))]=z\partial_z\rho(f(z))+\rho(f(z)),\quad  [L_{0},\rho(e(z))]=z\partial_z\rho(e(z))+\rho(e(z)),
$$
$$
[L_{0},\rho(h(z))]=z\partial_z\rho(h(z))+\rho(h(z)),\quad  [L_{0},b(z)]=z\partial_zb(z)+b(z),
$$
as operators on any imaginary Wakimoto module. 
\end{cor}

\section{Intertwining operators}

\subsection{Topological modules}
Following definitions given in \cite{War} page 22, we have:

\begin{definition}  A set $\mathcal N$ of subsets of a set $S$ is called a filter on $S$ if $S\in \mathcal N$, $\emptyset\not\in \mathcal N$, the intersection of two elements of $\mathcal N$ is again in $\mathcal  N$, and any subset of $S$ containing a set in $\mathcal N$ is again in $\mathcal N$.   Suppose $\mathcal B$ is a collection of subsets of a given set $S$ and consider the set $\mathcal N=\{F\subseteq  S\,|\, \exists B\in  \mathcal B,B\subseteq F\}$.  If $\mathcal N$ is a filter on $S$, then $\mathcal B$ is called a filter base on $S$.  A fundamental system of neighborhoods of zero in a vector space is any filter base generating the neighborhoods of zero. 

\end{definition}
\begin{thm}[\cite{War}, Theorem 3.5]

Suppose $A$ is a ring.  If $\mathcal N$ is a filter base of neighborhoods of zero for $A$, then $\mathcal N$ satisfies the following conditions :\\
(TRN1) For each $N\in\mathcal N$ there exists $U\in\mathcal N$ such that $U+U\subseteq N$.\\
(TRN2) $N\in\mathcal N$ implies there exists $U\in\mathcal N $ such that $U\subseteq -N$.\\
(TRN3) For each $N\in\mathcal N$ there exists $U\in\mathcal N$ such that $UU\subseteq N$.\\
(TRN4) For each $N\in\mathcal N$, $b\in A$ there exists $U\in\mathcal N$ such that $bU\subseteq N$ and $Ub\subseteq N$.\\
Conversely if $\mathcal N$ is a filter base on $A$  satisfying (TRN1)-(TRN4),  then there is a unique ring topology on $A$ for which $\mathcal N$ is a fundamental system of neighborhoods of zero.
\end{thm}

Now to define topological modules $M$ for a topological ring $A$ we need a collection of neighborhoods of zero $\mathfrak n$ of subsets $N$ of $M$ such that\\ \\
(M1) For each $N\in\mathfrak n$ there exists $U\in\mathfrak n$ such that $U+U\subseteq N$.\\
(M2) For each $N\in\mathfrak n$ there exists $U\in\mathfrak n$ such that $U\subseteq-N$.\\
(M3) For each $N\in\mathfrak n$ there exists a neighborhood $T$ of zero in $A$ and $U\in\mathfrak n$ such that $TU\subseteq N$.\\
(M4) For each $N\in\mathfrak n$ and each $b\in N$ there exists a neighborhood $T$ of zero in A such that $Tb\subseteq N$.\\
(M5) For each $N\in\mathfrak n$ and each $\beta\in A$ there exists $U\in\mathfrak n$ such that $\beta U\subseteq\mathfrak n$.\\

The set $\{N\}$ satisfying the above (M1)-(M5) forms a base for a topological module $M$. 

We want to find a completion of $M=U(\hat{\mathfrak g}_-)\cdot1\subseteq V_{\lambda-m}$ and $W_{\lambda-m}$ such that
$$
f_{m_1}\cdots f_{m_k}\Big(\displaystyle\sum_{n\in\mathbb N}\displaystyle\sum_{\pi\in P_n}\alpha_\pi h_\pi1\Big)\in\hat{V}_{\lambda-m}.
$$

Let
$$
U_n=span\langle h_{n_1}\cdots h_{n_k}|n\leqslant\displaystyle\sum_{i=1}^kn_i\rangle\subseteq A,\ \mathfrak n=\{U_n|n\in\mathbb N\}
$$
be a fundamental system of neighborhoods of zero for $A$,
$$
\mathfrak n'=\{U_n\cdot1|n\in\mathbb N\}\text{ topology for }M,
$$
$$
U_n\subseteq U(\hat{\mathfrak g}_-)=A,\ M_n=U_n\cdot1,\ M=U(\mathfrak g)\cdot1,
$$
\begin{equation}
\hat{M}:=\displaystyle\varprojlim_n M/M_n
\end{equation}
completion of $M$. So we have:\\ \\
(M1) If $N=M_n\in\mathfrak n$, then $U=M_n$ satisfies $M_n+M_n\subseteq M_n$.\\
(M2) If $N=M_n\in\mathfrak n$, then $U=M_n$ satisfies $M=M_n\subseteq-N$.\\
(M3) If $N=M_n\in\mathfrak n$, then for $T=U_n$ and $U=M_n\in\mathfrak n$ we have $TM=U_nU_n\cdot1\subseteq U_{2n}\cdot1\subseteq U_n\cdot1=M_n$.\\
(M4) For $N=M_k\in\mathfrak n$, $b=\displaystyle\sum_{n\in\mathbb N}\displaystyle\sum_{\pi\in P_n}\alpha_\pi h_\pi\cdot1$, then take $T=U_k$ and then we have $U_k b\subseteq M_k$.\\
(M5) For $N=M_k\in\mathfrak n$, $\beta=\displaystyle\sum_{n\in\mathbb N}\displaystyle\sum_{\pi\in P_n}\alpha_\pi h_\pi\in A$, take $U=M_k$ and we have $bM_k\subseteq M_k$.\\

We need to make sure $\hat{M}$ can be made into an $\hat{\mathfrak n}_-\oplus\hat{\mathfrak h}$-module. Now we have to check left multiplication by $h_n$, $n\in\mathbb N$ is continuous: $h_nM_m\subseteq M_{n+m}$.\\

So $h_n$ extends to a map $h_n:\hat{M}\longrightarrow\hat{M}$.

\begin{definition}
Define $\hat{V}_{\lambda,\kappa}=\displaystyle\varprojlim_n V_{\lambda,\kappa}/(V_{\lambda,\kappa})_n$ and similar to $\hat{W}_{\lambda,\kappa}$.
\end{definition}

\begin{definition} Let $\mathfrak F_m$ be the finite dimensional $\mathfrak{sl}(2,\mathbb C)$-module with highest weight $m\in\mathbb N$ and highest weight vector $u^m$.
\end{definition}

Observe that the definitions of $\mathfrak F_m(z)$ and $V_{\lambda-m,\kappa}\hat{\otimes}\mathfrak F_m(z)$ are similar to the definitions given in \eqnref{evaluationrepresentation} and \eqnref{completedtensorproduct}.  Set $u^m_k=u^m\otimes z^k\in \mathfrak F_m(z)$ for $k\in \mathbb Z$.

\begin{definition}
Define $P_n:=\{(-n_1,-n_2,\cdots,-n_p)|n_1+\cdots+n_p=n\}$ and for a fixed $\pi=(-n_1,\cdots,-n_p)\in P_n$, define $h_\pi:=h_{n_1}\cdots h_{n_p}$.
\end{definition}

\begin{prop}\label{singularvector}
Suppose $\kappa \neq 0$ and fix $\beta_{(1)}\in\mathbb C$.  Let $m\in\mathbb N$, $V_{\lambda-m,\kappa}$ the imaginary Verma module for $\hat{\mathfrak g}$ and
 $v_{\lambda-m,\kappa}\in V_{\lambda-m,\kappa}$ be an imaginary highest weight vector of highest weight $\lambda-m$.   
For $(n_1,n_2,\dots,n_r)=(n_1,n_2,\dots, n_r,0,0,\dots)$ with $n_i\in\mathbb N$ define 
$$
\beta_{(n_1,\dots,n_r)}=\dfrac{m^{n_r+\cdots+n_1-1}}{(-\kappa )^{n_r+\dots+n_1-1}(r^{n_r}\cdot\dots\cdot 2^{n_2})\cdot n_r!\dots n_1!}\cdot\beta_{(1)}
$$
For a partition 
$$
\pi=(\underbrace{-1,-1,\dots,-1}_{n_1},\underbrace{-2,-2,\dots,-2}_{n_2},\dots,\underbrace{-r,-r,\dots,-r}_{n_r})
$$
of $-n$ (so $\displaystyle\sum_{l=1}^rln_l=n$) define  
$$
\alpha_{\footnotesize{(\underbrace{-1,-1,\dots,-1}_{n_1},\underbrace{-2,-2,\dots,-2}_{n_2},\dots,\underbrace{-r,-r,\dots,-r}_{n_r})}}=\beta_{(n_1,n_2,\dots,n_r)}
$$
and 
$$
h_\pi:=h_{-1}^{n_1}h_{-2}^{n_2}\cdots h_{-r}^{n_r}
$$
The vector 
$$
v^\sharp _{\lambda,\kappa}:=\sum_{n\in\mathbb N}\sum_{\pi\in P_n}\alpha_\pi h_\pi v_{\lambda-m,\kappa}\otimes u^m_n\in V_{\lambda-m,\kappa}\hat\otimes \mathfrak F_m(z)
$$
is an imaginary highest weight vector of weight $\lambda$.  Hence there is a nonzero 
$\hat{\mathfrak g}$-module homomorphism $\Phi^V(z):V_{\lambda,\kappa} \to V_{\lambda-m,\kappa}\hat{\otimes} \mathfrak  F_m(z)$ such that $\Phi^V(z)(v_{\lambda,\kappa})=v^\sharp _{\lambda,\kappa}$.
\end{prop}
\begin{proof}
We write out an indication of the calculation in the case of the action of $h_{1}$: 
\begin{align*}
&h_1\Big(\alpha_0v_{\lambda-m,\kappa}\otimes u_0^m+\alpha_{(-1)}h_{-1}v_{\lambda-m,\kappa}\otimes u_1^m+\left(\alpha_{(-1,-1)}h_{-1}^2+
\alpha_{(-2)}h_{-2}\right)v_{\lambda-m,\kappa}\otimes u_2^m\\
&\quad +\left(\alpha_{(-1,-1,-1)}h_{-1}^3 +\alpha_{(-2,-1)}h_{-1} h_{-2}+\alpha_{(-3)}h_{-3}\right)v_{\lambda-m,\kappa}\otimes u_3^m\\
&\quad +\big(\alpha_{(-1,-1,-1,-1)}h_{-1}^4 +\alpha_{(-2,-1,-1)}h_{-1}^2h_{-2}  +\alpha_{(-2,-2)}h_{-2}^2\\
&\quad+\alpha_{(-3,-1)}h_{-1}h_{-3}+\alpha_{(-4)}h_{-4}\big)v_{\lambda-m,\kappa}\otimes u_4^m+\cdots\Big)  \\
&=0+m\alpha_0 v_{\lambda-m,\kappa}\otimes u_1^m+\alpha_{(-1)}\kappa v_{\lambda-m,\kappa}\otimes u_1^m+m\alpha_{(-1)}h_{-1}v_{\lambda-m,\kappa}\otimes u_2^m\\
&\quad+2\alpha_{(-1,-1)}\kappa h_{-1}v_{\lambda-m,\kappa}\otimes u_2^m+m\alpha_{(-1,-1)}h_{-1}^2v_{\lambda-m,\kappa}\otimes u_3^m+m\alpha_{(-2)}h_{-2}v_{\lambda-m,\kappa}\otimes u_3^m\\
&\quad+3\alpha_{(-1,-1,-1)}\kappa h_{-1}^2v_{\lambda-m,\kappa}\otimes u_3^m+m\alpha_{(-1,-1,-1)}h_{-1}^3v_{\lambda-m,\kappa}\otimes u_4^m \\
&\quad +\alpha_{(-2,-1)}\kappa h_{-2}v_{\lambda-m,\kappa}\otimes u_3^m+m\alpha_{(-2,-1)}h_{-1}h_{-2}v_{\lambda-m,\kappa}\otimes u_4^m+m\alpha_{(-3)}h_{-3}v_{\lambda-m,\kappa}\otimes u_4^m+ \\
&\quad +\kappa \left(4\alpha_{(-1,-1,-1,-1)}h_{-1}^3 +2\alpha_{(-2,-1,-1)}h_{-1}h_{-2} + \alpha_{(-3,-1)}h_{-3}\right)v_{\lambda-m,\kappa}\otimes u_4^m +\cdots\\
\end{align*}

 For $1\leqslant k\leqslant r$, we have\\ \\
$\begin{aligned}
h_k\big(\beta_{(n_1,n_2,\dots,n_r)}&h_{-1}^{n_1}\dots h_{-k}^{n_k}\dots h_{-r}^{n_r}v_{\lambda-m,\kappa}\otimes u^m_n\big)\\
&=k.n_k.\kappa .\beta_{(n_1,\dots,n_k,\dots,n_r)}h_{-1}^{n_1}\dots h_{-k}^{n_k-1}\dots h_{-r}^{n_r}v_{\lambda-m,\kappa}\otimes u^m_n\\
&\quad+m\beta_{(n_1,\dots,n_k,\dots,n_r)}h_{-1}^{n_1}\dots h_{-k}^{n_k}\dots h_{-r}^{n_r}v_{\lambda-m,\kappa}\otimes u^m_{n+k}\\
\end{aligned}$\\ \\

For this reason we want in general to have
\begin{equation}\label{beta}
m\beta_{(n_1,\dots,n_k-1,\dots,n_r)}+k.n_k.\kappa \beta_{(n_1,\dots,n_k,\dots,n_r)}=0.
\end{equation}
In this case when we apply $h_k$, the coefficient of $h_{-1}^{n_1}\dots h_{-k}^{n_k-1}\dots h_{-r}^{n_r}v_{\lambda-m,\kappa}\otimes u_n^m$, 
where $n=\sum_{i=1}^rin_i$, is $m\beta_{(n_1,\dots,n_k-1,\dots,n_r)}+k.n_k.\kappa \beta_{(n_1,\dots,n_k,\dots,n_r)}$ and we want it to be zero. Then we can fix $\beta_{(1)}$ and find all the 
others using this condition:
\begin{align*}
m\beta_{(1)}+l.\kappa \beta_{(1,0,\dots,0,1)}=0\quad \Longrightarrow\quad \beta_{(1,0,\dots,0,1)}=\frac{-m}{l\cdot\kappa }\beta_{(1)}\\
m\beta_{(0,\dots, 0,1)}+1.\kappa \beta_{(1,0,\dots,0,1)}=0\quad \Longrightarrow\quad \beta_{(0,0,\dots,0,1)}=\frac{1}{l}\beta_{(1)},
\end{align*}
where $1$ occurs in the (last) $l$-entry of $(1,0,\dots,0,1)$ and $(0,\dots, 0,1)$.
Similarly by applying $h_k$ repeatedly we need to require
\begin{align*}
m^{n_k}\beta_{(n_1,\dots,n_{k-1},0,n_{k+1},\dots,n_r)}=(-1)^{n_k}\cdot k^{n_k}\cdot\kappa ^{n_k}\cdot n_k!\beta_{(n_1,\dots,n_k,\dots,n_r)}
\end{align*}

Without lost of generality, we can assume $n_1\neq0$ and require
\begin{align*}
\beta_{(1)}=\dfrac{(-\kappa )^{n_r+\dots+n_1-1}\cdot(r^{n_r}\dots 2^{n_2})\cdot n_r!\dots n_1!}{m^{n_r+\cdots+n_1-1}}\cdot\beta_{(n_1,\dots,n_r)}
\end{align*}

So under the hypothesis $\kappa \neq 0$,
\begin{align*}
\beta_{(n_1,\dots,n_r)}=\dfrac{m^{n_r+\cdots+n_1-1}}{(-\kappa )^{n_r+\dots+n_1-1}(r^{n_r}\dots 2^{n_2})\cdot n_r!\dots n_1!}\cdot\beta_{(1)}
\end{align*}
is well defined and forces $v^\sharp_{\lambda,\kappa}$ to be annihilated by the $h_k$ for all $k\geq 1$.

Now we have
\begin{align*}
h_m(v^\sharp_{\lambda,\kappa})&=0,\\
e_m(v^\sharp_{\lambda,\kappa})&=\displaystyle\sum_{n\in\mathbb N^*}\displaystyle\sum_{\pi\in P_n}\alpha_\pi(e_mh_\pi v_{\lambda-m,\kappa})\otimes u^m_n+\displaystyle\sum_{n\in\mathbb N^*}\displaystyle\sum_{\pi\in P_n}\alpha_\pi h_\pi v_{\lambda-m,\kappa}\otimes e_mu^m_n\\
&=0
\end{align*}
for all $m>0$
(because $e_mh_k=-2e_{m+k}+h_ke_m,\ \forall k$).  Thus $v^\sharp_{\lambda,\kappa}$ is an imaginary highest weight vector of weight $\lambda+\kappa \Lambda_0$. By the universal mapping property of 
imaginary Verma modules, there exists a $\hat{\mathfrak g}$-module homomorphism $\varphi: V_{\lambda,\kappa}\longrightarrow V_{\lambda-m,\kappa}\hat{\otimes} \mathfrak F_m(z)$, 
for $\kappa \neq 0$, sending $v_{\lambda,\kappa}$ to $v_{\lambda,\kappa}^\sharp$.

\end{proof}

\begin{cor}
\begin{align} 
 ((zh(z))_- \otimes 1 )\hat\Phi^V(z)(v_{\lambda,\kappa})
 &=-m \hat\Phi^V(z)(v_{\lambda,\kappa}).
\end{align}

\end{cor}
\begin{proof} By \eqnref{beta} we have 
\begin{align*}
(h_k\otimes 1)\big(\beta_{(n_1,n_2,\dots,n_r)}&h_{-1}^{n_1}\dots h_{-k}^{n_k}\dots h_{-r}^{n_r}v_{\lambda-m,\kappa}\otimes u^m_n\big)\\
&=k.n_k.\kappa .\beta_{(n_1,\dots,n_k,\dots,n_r)}h_{-1}^{n_1}\dots h_{-k}^{n_k-1}\dots h_{-r}^{n_r}v_{\lambda-m,\kappa}\otimes u^m_n\\
&=-m\beta_{(n_1,\dots,n_k-1,\dots,n_r)}h_{-1}^{n_1}\dots h_{-k}^{n_k-1}\dots h_{-r}^{n_r}v_{\lambda-m,\kappa}\otimes u^m_n
\end{align*} 
and hence
\begin{align} 
((z h(z))_-\otimes  1)\hat\Phi^V(z)(v_{\lambda,\kappa})
 &=\sum_{k>0} \sum_{n\in\mathbb N}\sum_{\pi\in P_n}\alpha_\pi [h_k,h_\pi] v_{\lambda-m,\kappa}\otimes z^{-\Delta}u^m_nz^{-k} \notag  \\ 
 &=-m\sum_{k>0} \sum_{n\in\mathbb N}\sum_{\pi\in P_n}\alpha_\pi  h_\pi v_{\lambda-m,\kappa}\otimes z^{-\Delta-k}u^m_{n+k} \notag  \\ 
 &=-mz^{-\Delta}  \sum_{n\in\mathbb N}\sum_{\pi\in P_n}\alpha_\pi  h_\pi v_{\lambda-m,\kappa}\otimes  u^m_n \notag  \\ 
 &=-m \hat\Phi^V(z)(v_{\lambda,\kappa}) .\notag 
\end{align}

\end{proof}

In particular note the intertwining property of $\Phi^V(z)$: If we let $*$ denote the tensor product action, then
\begin{equation}\label{intertwining} 
\Phi^V(z)x_n=x_n*\Phi(z)=(x_n\otimes 1+z^n(1\otimes x))\Phi^V(z)
\end{equation}
for $x\in\mathfrak{sl}(2,\mathbb C)$.

\subsection{Operator Form of the KZ Equation}
We define  $\Phi^W(z):W_{\lambda,\kappa} \to W_{\lambda-m,\kappa}\hat{\otimes} \mathfrak  F_m(z)$ through the diagram
$$
\begin{CD}  W_{\lambda,\kappa}  @>\Phi^W(z)>> W_{\lambda-m,\kappa}\hat{\otimes} \mathfrak  F_m(z)  \\
@V\Psi_{\lambda,\kappa}^{-1}VV @A \Psi_{\lambda-m,\kappa} \otimes 1 AA \\
V _{\lambda,\kappa} @>\Phi^V(z)>>V_{\lambda-m,\kappa}\hat{\otimes} \mathfrak  F_m(z)  \\
\end{CD}
$$
where $\Psi_{\lambda,\kappa}:V_{\lambda,\kappa} \to W_{\lambda,\kappa}$ is the canonical isomorphism given in 
\thmref{cf05} and
$$
(\Psi_{\lambda-m,\kappa}\otimes1)(v_1\otimes v_2):=\Psi_{\lambda-m,\kappa}(v_1) \otimes v_2.
$$

Then
\begin{equation}\label{phiw}
\Phi^W(z):=( \Psi_{\lambda -m,\kappa} \otimes 1)\circ \Phi^V(z)\circ \Psi_{\lambda,\kappa}^{-1}.
\end{equation}

Now consider $z$ as a formal variable and set
$$
\mathfrak F_m\otimes z^{-\Delta}\mathbb C[z,z^{-1}]=z^{-\Delta}\mathfrak F_m[z_1,z_1^{-1}]. 
$$
This space is an infinite dimensional representation of $\hat{\mathfrak g}$ with a basis $v\otimes z^{n-\Delta},\quad n\in\mathbb Z,\quad v\in\mathfrak F_m$.


\begin{prop}\label{iff}
The $\hat{\mathfrak g}$-module map $z^{-\Delta}\Phi^W(z):W_{\lambda,\kappa}\longrightarrow W_{\lambda-m,\kappa}\hat{\otimes}z^{-\Delta}\mathfrak F_m[z,z^{-1}]$ such that 
$$
z^{-\Delta}\Phi^W(z)(w_{\lambda,\kappa})=\displaystyle\sum_{n\in\mathbb N^*}\displaystyle\sum_{\pi\in P_n}\alpha_\pi b_\pi w_{\lambda-m,\kappa}\otimes z^{-\Delta}u^m_n,
$$
is a $\tilde{\mathfrak g}$-module 
homomorphism (here $d$ acts by $ z\frac{\partial}{\partial z}$ on the second factor)

if and only if  
\begin{equation}\label{Delta}
\Delta=\Delta(\lambda)-\Delta(\lambda-m).
\end{equation}
\end{prop}

\begin{proof}
 By the definition we have $d\cdot w_{\lambda,\kappa}=-\Delta(\lambda)\cdot w_{\lambda,\kappa}$, for $w_{\lambda,\kappa}\in W_{\lambda,\kappa}$ (see \eqnref{vlk}). Now $z^{-\Delta}\Phi^W(z)$ is a $\tilde{\mathfrak g}$-intertwiner if and only if it satisfies
\begin{equation}\label{intertwiner}
z^{-\Delta}\Phi^W(z)d-(d\otimes1)z^{-\Delta}\Phi^W(z)=\Big(1\otimes z\frac{\partial}{\partial z}\Big)z^{-\Delta}\Phi^W(z).
\end{equation}

We will apply both sides of this equality to the imaginary highest weight vector $w_{\lambda,\kappa}$ and see that they agree if and only if the condition \eqnref{Delta} is satisfied.  Then since $W_{\lambda,\kappa}$ is generated by this vector the two sides agree on all of $W_{\lambda,\kappa}$ if and only if \eqnref{Delta} is satisfied.   

The left hand side of \eqnref{intertwiner} applied to $w_{\lambda,\kappa}$ gives us


\begin{align*}
\left(z^{-\Delta}\Phi^W(z)d-(d\otimes 1)z^{-\Delta }\Phi^W(z)\right)(w_{\lambda,\kappa})
&=-\Delta(\lambda)\sum_{n\in\mathbb N^*}\displaystyle\sum_{\pi\in P_n}\alpha_\pi b_\pi w_{\lambda-m,\kappa}\otimes u^m_nz^{-\Delta} \\
&\quad -\sum_{n\in\mathbb N^*}\displaystyle\sum_{\pi\in P_n}\alpha_\pi(-n-\Delta(\lambda-m)) b_\pi w_{\lambda-m,\kappa}\otimes u^m_nz^{-\Delta}  \\
&=-\left(\Delta(\lambda)-\Delta(\lambda-m)\right)\sum_{n\in\mathbb N^*}\displaystyle\sum_{\pi\in P_n}\alpha_\pi b_\pi w_{\lambda-m,\kappa}\otimes u^m_nz^{-\Delta} \\
&\quad +\sum_{n\in\mathbb N^*}\displaystyle\sum_{\pi\in P_n}n\alpha_\pi   b_\pi w_{\lambda-m,\kappa}\otimes u^m_nz^{-\Delta}
\end{align*}
while the right hand side is
\begin{align*}
\Big(1\otimes z\frac{\partial}{\partial z}\Big)z^{-\Delta}\Phi^W(z)(w_{\lambda,\kappa})
&= \sum_{\pi\in P_n}\alpha_\pi b_\pi w_{\lambda-m,\kappa}\otimes z\frac{\partial}{\partial z}u^m_nz^{-\Delta} \\
&=-\Delta \sum_{\pi\in P_n}\alpha_\pi b_\pi w_{\lambda-m,\kappa}\otimes u^m_nz^{-\Delta}   \\
&\quad +\sum_{\pi\in P_n}n\alpha_\pi  b_\pi w_{\lambda-m,\kappa}\otimes u^m_nz^{-\Delta}.
\end{align*}

To finish the proof we recall \eqnref{phiw} and note that $\Psi_{\lambda,\kappa}$ is a $\tilde{\mathfrak g}$-module homomorphism.

\end{proof}


For $x\in\mathfrak F_m^*$,  define $\Phi_x^V(z): V_{\lambda,\kappa}\longrightarrow\hat{V}_{\lambda-m,\kappa}$ 
 by
 \begin{equation}
\Phi_x^V(z)(w)=x(\Phi^V(z)(w)).
 \end{equation}
For example
\begin{equation}
\Phi_x^V(z)(v_{\lambda,\kappa})=\displaystyle\sum_{n\in\mathbb N^*}\displaystyle\sum_{\pi\in P_n}\alpha_\pi h_\pi v_{\lambda-m,\kappa} x(u_n^m),
\end{equation}  where $xu_n^m=(xu^m) z^n$.
We similarly define $\Phi_x^W(z):W_\lambda\to \hat W_{\lambda -m,\kappa}$. The intertwining property \eqnref{intertwining} for $x=f$ becomes
\begin{align*}
 \Phi^W_x(z)a_m &= \Phi^W_x(z)f_m=f_m*  \Phi^W_x(z) =(f_m\otimes 1) \Phi^W_x(z)+z^m(1\otimes f) \Phi^V_x(z) \\
&=(a_m\otimes 1) \Phi^W_x(z)-z^m  \Phi^W_{fx}(z). \\
\end{align*}
where $fx$ means the action of $f$ on an element in the dual space $\mathfrak F_m^*$.
Thus
\begin{equation}\label{am}
[a_m, \Phi^W_x(z)]=z^m  \Phi^W_{fx}(z).
\end{equation}
A similar computation shows
\begin{equation}\label{bm}
[b_m, \Phi^W_x(z)]=z^m  \Phi^W_{hx}(z).
\end{equation}
The equations \eqnref{am} and \eqnref{bm} tell us 
\begin{equation}
[b(\zeta)_+,\Phi_x^W(z)]=\frac{1}{z-\zeta}\Phi_{hx}^W(z),\qquad [b(\zeta)_-,\Phi_x^W(z)]=\frac{1}{\zeta-z}\Phi_{hx}^W(z),\qquad  [a(\zeta),\Phi_x^W(z)]=\Phi_{fx}^W(z)\delta(\zeta/z).
\end{equation}
%
%

For $x\in \mathfrak F_m^*$ and  $\Delta=\Delta(\lambda)-\Delta(\lambda -m)$ set $\tilde\Phi^W_x(z):=z^{-\Delta}\Phi^W_x(z)$. 
%
%
%
%

\begin{thm}[Operator form of the KZ-Equations]
For $x\in\mathfrak F_m^*$ of weight $\alpha\in\mathbb C$, set
$$
\Delta(\alpha,\mu):=\frac{\alpha(\alpha-2\mu)}{4}\quad \text{and}\quad \hat\Phi^W_x(z):=z^{-\Delta(\mu,\alpha)}\tilde\Phi^W_x(z),
$$
where $\mu$ is fixed (see \secref{mu}). The operator  $\hat{\Phi}_x^W(z):W_{\lambda,\kappa}\longrightarrow W_{\lambda-m,\kappa}\hat{\otimes}\mathfrak F_m[z,z^{-1}]\cdot z^{-\Delta}$ 
satisfies the differential equation
\begin{equation}
\boxed{
\frac{d}{dz}\hat\Phi^W_x(z)=\hat\Phi_{fx}^W(z) \partial_za^*(z)  +  \dfrac{\alpha}{2} :b(z) \hat\Phi_{x}^W(z):  }
\end{equation}
\end{thm}

\begin{proof}

We have
\begin{align*}
\hat\Phi^W_x(z)&\circ a_n^*(a_{n_1}\cdots a_{n_k}b_{-m_1}\cdots b_{-m_l}\cdot w_{\lambda,\kappa}) \\
&=-\sum_{i=1}^k\delta_{n+n_i,0} \hat\Phi^W_x(z)\left( a_{n_1}\cdots a_{n_{i-1}}a_{n_{i+1}}\cdots a_{n_k}b_{-m_1}\cdots b_{-m_l}\cdot w_{\lambda,\kappa}) \right),
\end{align*}
while
\begin{align*}
a_n^*\circ &\hat\Phi^W(z)(a_{n_1}\cdots a_{n_k}b_{-m_1}\cdots b_{-m_l}\cdot w_{\lambda,\kappa})\\
&=a_n^*\circ (a_{n_1}*\cdots* a_{n_k}*b_{-m_1}*\cdots * b_{-m_l}* \hat\Phi^W(z)(w_{\lambda,\kappa}))\\
&=a_n^*  ((a_{n_1}\otimes1+z^{n_1}\otimes f)\cdots (a_{n_k}\otimes1+z^{n_k}\otimes f)  \\
&\hskip 100pt\times (b_{-m_1}\otimes1+z^{-m_1}\otimes h)\cdots (b_{-m_l}\otimes1+z^{-m_l}\otimes h)) \hat\Phi^W(z)(w_{\lambda,\kappa})\\
&= -\sum_{i=1}^k \delta_{n+n_i,0}(a_{n_1}\otimes1+z^{n_1}\otimes f)\cdots(a_{n_{i-1}}\otimes1+z^{n_{i-1}}\otimes f)(a_{n_{i+1}}\otimes1+z^{n_{i+1}}\otimes f)\cdots(a_{n_k}\otimes1+z^{n_k}\otimes f)\\
&\quad\cdot(b_{-m_1}\otimes1+z^{-m_1}\otimes h)\cdots(b_{-m_l}\otimes1+z^{-m_l}\otimes h)\cdot \hat\Phi^W(z)(w_{\lambda,\kappa})\\
&\quad +a_{n_1}*\cdots* a_{n_k}*b_{-m_1}*\cdots * b_{-m_l}* (a_n^*\otimes 1)\hat\Phi^W(z)(w_{\lambda,\kappa}))\\
&= -\sum_{i=1}^k \delta_{n+n_i,0}\hat\Phi^W(z)(a_{n_1}\cdots a_{n_{i-1}}a_{n_{i+1}}\cdots a_{n_k}b_{-m_1}\cdots b_{-m_l}\cdot w_{\lambda,\kappa})\\
&\quad +a_{n_1}*\cdots* a_{n_k}*b_{-m_1}*\cdots * b_{-m_l}* (a_n^*\otimes 1)\hat\Phi^W(z)(w_{\lambda,\kappa})).\\
\end{align*}
On the other hand\\
\begin{align}
a_n^*\hat\Phi^W(z)(w_{\lambda,\kappa})&= \sum_{n\in\mathbb N}\sum_{\pi\in P_n}\alpha_\pi a_n^*b_\pi w_{\lambda-m,\kappa}\otimes z^{-\Delta}u^m_n  =0,\notag
\end{align}
so that \\
\begin{align*}
 (a_n^*\otimes 1)\circ &\hat\Phi^W_x(z)(a_{n_1}\cdots a_{n_k}b_{-m_1}\cdots b_{-m_l}\cdot w_{\lambda,\kappa})\\
&=(1\otimes x)\circ (a_n^*\otimes 1)\circ \hat\Phi^W(z)(a_{n_1}\cdots a_{n_k}b_{-m_1}\cdots b_{-m_l}\cdot w_{\lambda,\kappa})\\
&= -\sum_{i=1}^k \delta_{n+n_i,0}(1\otimes x)\hat\Phi^W(z)(a_{n_1}\cdots a_{n_{i-1}}a_{n_{i+1}}\cdots a_{n_k}b_{-m_1}\cdots b_{-m_l}\cdot w_{\lambda,\kappa})  \\
&= -\sum_{i=1}^k \delta_{n+n_i,0} \hat\Phi_x^W(z)(a_{n_1}\cdots a_{n_{i-1}}a_{n_{i+1}}\cdots a_{n_k}b_{-m_1}\cdots b_{-m_l}\cdot w_{\lambda,\kappa}).
\end{align*}
Thus $[a_n^*, \hat\Phi_x^W(z)]=0$. We have\\
$$
z\frac{d}{dz}\hat\Phi_x^W(z)=-[d,\tilde\Phi_x^W(z)]-\Delta(\mu,\alpha)\hat\Phi_x^W(z).
$$\\
Replace $d$ by $-L_0$ and use \eqnref{L}:\\
\begin{align*}
z\frac{d}{dz}\hat\Phi_x^W(z)&=[L_0,\hat\Phi^W_x(z)]-\Delta(\mu,\alpha)\hat\Phi_x^W(z) \\
&=\left[\sum_{n\in\mathbb Z}na_na^*_{-n}+\dfrac{1}{4}\left(\displaystyle\sum_{n\in\mathbb Z}:b_nb_{-n}:\right)-\frac{\mu}{2} b_0 ,\hat\Phi_x^W(z)\right]-\Delta(\mu,\alpha)\hat\Phi_x^W(z) \\
&=\sum_{n\in\mathbb Z}n\left[a_n,\hat\Phi_x^W(z)\right]a^*_{-n} \\
&\quad +\left[ \dfrac{1}{4}\left( \sum_{n>0}b_{-n}b_n\right)+ \dfrac{1}{4}\left( \sum_{n<0}b_nb_{-n}\right)+\frac{1}{4}b_0^2-\frac{\mu}{2} b_0 ,\hat\Phi_x^W(z)\right] -\Delta(\mu,\alpha)\hat\Phi_x^W(z)\\  \\
&=\hat\Phi_{fx}^W(z) \sum_{n\in\mathbb Z}na^*_{-n}z^n \\
&\quad +  \dfrac{1}{2}\left( \sum_{n>0}\left[b_{-n},\hat\Phi_x^W(z)\right] b_n\right)+ \dfrac{1}{2}\left( \sum_{n>0}b_{-n}\left[b_n,\hat\Phi_x^W(z)\right] \right)+\frac{1}{4}\left(\left[b_0,\hat\Phi_x^W(z)\right] b_0+b_0\left[b_0,\hat\Phi_x^W(z)\right] \right) \\
&\hskip 100pt-\frac{\mu}{2}\left[ b_0,\hat\Phi_x^W(z)\right] -\Delta(\mu,\alpha)\hat\Phi_x^W(z)\\  \\
&=\hat\Phi_{fx}^W(z) z\partial_za^*(z)\\
&\quad +  \dfrac{1}{2}\left( \sum_{n>0} \hat\Phi_{hx}^W(z)  b_nz^{-n}\right)+ \dfrac{1}{2}\left( \sum_{n>0}b_{-n} \Phi_{hx}^W(z)z^n \right)+\frac{1}{4}\left( \hat\Phi_{hx}^W(z) b_0+b_0 \hat\Phi_{hx}^W(z)  \right) \\
&\hskip 100pt-\frac{\mu}{2} \hat\Phi_{hx}^W(z) -\Delta(\mu,\alpha)\hat\Phi_x^W(z)\\  \\
&=\hat\Phi_{fx}^W(z) z\partial_za^*(z)\\
&\quad +  \dfrac{z}{2} :b(z) \hat\Phi_{hx}^W(z):  +\frac{1}{4}\left( -\hat\Phi_{hx}^W(z) b_0+b_0 \hat\Phi_{hx}^W(z)  \right)  -\frac{\mu}{2} \hat\Phi_{hx}^W(z) -\Delta(\mu,\alpha)\hat\Phi_x^W(z)\\  \\
&=\hat\Phi_{fx}^W(z) z\partial_za^*(z)  +  \dfrac{z}{2} :b(z) \hat\Phi_{hx}^W(z):  +\frac{1}{4}\hat\Phi_{h^2x}^W(z)   -\frac{\mu}{2} \hat\Phi_{hx}^W(z) -\Delta(\mu,\alpha)\hat\Phi_x^W(z)\\  \\
&=\hat\Phi_{fx}^W(z) z\partial_za^*(z)  +  \dfrac{\alpha z}{2} :b(z) \hat\Phi_{x}^W(z):  .
\end{align*}

\end{proof}

\begin{thm}

\begin{align*} [L_m,\hat\Phi^W_x(z)]=\hat\Phi_{fx}^W(z) z^{m+1}\partial_za^*(z) +\dfrac{\alpha}{2}z^{m+1}:b(z)\hat\Phi^W_{x}(z):+(m+1)\Delta(\mu,\alpha)z^m\hat\Phi^W_{x}(z).
\end{align*}

\end{thm}

\begin{proof} We have
\begin{align*}
\ [L_m,\hat\Phi^W_x(z)]&=\left[\sum_{n\in\mathbb Z}(n-m)a_na^*_{m-n}+\dfrac{1}{4}\left(\displaystyle\sum_{n\in\mathbb Z}:b_nb_{m-n}:\right)-\frac{\mu}{2}(m+1) b_m,\hat\Phi_x^W(z)\right]\\
&=\sum_{n\in\mathbb Z}(n-m)\left(\left[a_n,\hat\Phi^W_x(z)\right]a^*_{m-n}+a_n\left[a_{m-n}^*,\hat\Phi^W_x(z)\right]\right)\\
&\quad+\dfrac{1}{4}\left[\displaystyle\sum_{n\in\mathbb Z}:b_nb_{m-n}:,\hat\Phi^W_x(z)\right]-\frac{\mu}{2}(m+1) \left[b_m,\hat\Phi_x^W(z)\right]\\
&=\hat\Phi_{fx}^W(z) \sum_{n\in\mathbb Z}(n-m)a^*_{m-n}z^n-\frac{\mu}{2}(m+1)\left[ b_m,\hat\Phi_x^W(z)\right]  \\
&\quad+\dfrac{1}{4}\left[\displaystyle\sum_{n\in\mathbb Z}:b_nb_{m-n}:,\hat\Phi^W_x(z)\right]\\
&=\hat\Phi_{fx}^W(z) z^{m+1}\partial_za^*(z)-\frac{\mu}{2}(m+1)z^m\hat\Phi_{hx}^W(z)+\dfrac{1}{4}\left[\displaystyle\sum_{n\in\mathbb Z}:b_nb_{m-n}:,\hat\Phi^W_x(z)\right]\\
&=\hat\Phi_{fx}^W(z) z^{m+1}\partial_za^*(z)-\frac{\mu\alpha}{2}(m+1)z^m\hat\Phi_{x}^W(z)+\dfrac{1}{4}\left[\displaystyle\sum_{n\in\mathbb Z}:b_nb_{m-n}:,\hat\Phi^W_x(z)\right].
\end{align*}

So we only need to calculate $\left[\displaystyle\sum_{n\in\mathbb Z}:b_nb_{m-n}:,\hat\Phi^W_x(z)\right]$. Then we have\\ \\
$\begin{aligned}
\displaystyle\sum_{n\in\mathbb Z}&\left[:b_nb_{m-n}:,\hat\Phi^W_x(z)\right]=2\displaystyle\sum_{n>m}\left[b_{m-n}b_n,\hat\Phi^W_x(z)\right]+\displaystyle\sum_{i=0}^m\left[b_{m-i}b_i,\hat\Phi^W_x(z)\right]\\
&=2\displaystyle\sum_{n>m}b_{m-n}\left[b_n,\hat\Phi^W_x(z)\right]+2\displaystyle\sum_{n>m}\left[b_{m-n},\hat\Phi^W_x(z)\right]b_n\\
&\quad+\displaystyle\sum_{i=0}^mb_{m-i}\left[b_i,\hat\Phi^W_x(z)\right]+\displaystyle\sum_{i=0}^m\left[b_{m-i},\hat\Phi^W_x(z)\right]b_i\\
&=2z^{m+1}\Big(\displaystyle\sum_{n>m}b_{m-n}z^{n-m-1}\hat\Phi^W_{hx}(z)+\hat\Phi^W_{hx}(z)\displaystyle\sum_{n>m}b_nz^{-n-1}\Big)\\
&\quad+\Big(\displaystyle\sum_{i=0}^mb_{m-i}z^{i}\hat\Phi^W_{hx}(z)+\hat\Phi^W_{hx}(z)\displaystyle\sum_{i=0}^mb_iz^{m-i}\Big)\\
&=2z^{m+1}\Big(\displaystyle\sum_{n<0}b_{n}z^{-n-1}\hat\Phi^W_{hx}(z)+\hat\Phi^W_{hx}(z)\displaystyle\sum_{n\geq0}b_nz^{-n-1}\Big)-2z^{m+1}\hat\Phi^W_{hx}(z)\displaystyle\sum_{n=0}^mb_nz^{-n-1}\\
&\quad+z^{m+1}\displaystyle\sum_{i=0}^mb_{m-i}z^{-i-1}\hat\Phi^W_{hx}(z)+z^{m+1}\hat\Phi^W_{hx}(z)\displaystyle\sum_{i=0}^mb_iz^{-i-1}\\
&=2z^{m+1}:b(z)\hat\Phi^W_{hx}(z):+\displaystyle\sum_{i=0}^mb_{m-i}z^{m-i}\hat\Phi^W_{hx}(z)-\hat\Phi^W_{hx}(z)\displaystyle\sum_{i=0}^mb_iz^{m-i}\\
&=2z^{m+1}:b(z)\hat\Phi^W_{hx}(z):+\displaystyle\sum_{i=0}^mz^{m-i}[b_i,\hat\Phi^W_{hx}(z)]\\
&=2z^{m+1}:b(z)\hat\Phi^W_{hx}(z):+\displaystyle\sum_{i=0}^mz^m\hat\Phi^W_{h^2x}(z)\\
&=2z^{m+1}:b(z)\hat\Phi^W_{hx}(z):+(m+1)z^m\hat\Phi^W_{h^2x}(z)\\
&=2\alpha z^{m+1}:b(z)\hat\Phi^W_{x}(z):+(m+1)\alpha^2z^m\hat\Phi^W_{x}(z).
\end{aligned}
$\\ \\ \\
 Hence\\
\begin{align*} [L_m,\hat\Phi^W_x(z)]=\hat\Phi_{fx}^W(z) z^{m+1}\partial_za^*(z) +\dfrac{\alpha}{2}z^{m+1}:b(z)\hat\Phi^W_{x}(z):+(m+1)\Delta(\mu,\alpha)z^m\hat\Phi^W_{x}(z).
\end{align*}

\end{proof}

\section{Acknowledgement}
This work was initiated during the visit of the first author to University of S\~ao Paulo in  2010 and was finished during the visit of the third author to College of Charleston and of 
the second author to Max Planck Institute in Bonn in 2011. 
The first author would like to thank the Fapesp for support and the University of S\~ao Paulo for hospitality.  
The second author was supported in part by the CNPq grant
(301743/2007-0) and by the Fapesp grant (2010/50347-9). 
The third author was supported by Fapesp (2008/06860-3).

\end{document}